\documentclass{amsart}
\usepackage{amsmath}
\usepackage{amssymb}
\usepackage{amsthm}
\usepackage{mathtools}
\usepackage{amstext}
\usepackage{bm}
\usepackage[letterpaper]{geometry}

\usepackage{mathrsfs}

\usepackage{cancel}
\usepackage{nccmath}
\usepackage{xcolor}
\usepackage{graphicx}
\usepackage{setspace}

\usepackage{graphicx}
\usepackage{multirow}
\newtheorem{thm}{Theorem}[section]
\newtheorem*{thm*}{Theorem}
\newtheorem*{thmstar*}{Theorem*}
\newtheorem{cor}[thm]{Corollary}
\newtheorem{lem}[thm]{Lemma}
\newtheorem{prop}[thm]{Proposition}

\newtheorem*{Ques*}{Question}
\theoremstyle{definition}
\newtheorem{defn}[thm]{Definition}
\newtheorem*{claim}{Claim}
\newtheorem*{egs}{Examples}
\newtheorem*{eg}{Example}

\theoremstyle{remark}
\newtheorem{rem}[thm]{Remark}

\numberwithin{equation}{section}

\newcommand{\eps}{\varepsilon}

\newcommand{\fkg}{\mathfrak g}

\newcommand{\mcA}{\mathcal A}

\newcommand{\dbO}{{\mathbb O}}

\newcommand{\dbR}{{\mathbb R}}

\newcommand{\dbZ}{{\mathbb Z}}
\newcommand{\biga}{{\mathcal A}}

\newcommand{\Hom}{\mathrm{Hom}}

\newcommand{\la}{\langle}
\newcommand{\ra}{\rangle}

\DeclareMathOperator{\SL}{SL}
\DeclareMathOperator{\St}{St}

\DeclareMathOperator{\ad}{ad}

\DeclareMathOperator{\End}{End}

\DeclareMathOperator{\sgn}{sgn}

\usepackage{color}
\usepackage[normalem]{ulem}

\begin{document}

\title[]{Nonassociative algebras and groups with  property ($T$)}%
\author{Zezhou Zhang}%
\address{University of California, San Diego}%
\email{z9zhang@ucsd.edu}%

\thanks{}%

\begin{abstract}
We extend the results of \cite{EJK} on property (T) to certain groups "coordinatized" by nonassociative algebras.
\end{abstract}
\maketitle


\setcounter{section}{-1}
\section{Introduction}

Throughout this paper, unless otherwise stated,  \textit{rings} are always  unital,associative or alternative. \textit{Groups} are discrete, and group conjugation is defined to be $g^h=h^{-1}gh$.

It has been observed that a large class of groups can be characterized by root systems $\Phi$, namely the ones with ``$\Phi$-commutator relations'':

\begin{defn}[\cite{EJK}]
Let $G$ be a group and $\Phi$ a classical root system.Suppose that G has a family of
subgroups $\{X_\alpha\}_{\alpha \in \Phi}$ such that
\[
[X_\alpha , X_\beta] \ \ \subseteq
\prod_{\gamma \in \Phi \cap (\dbZ_{>0} \oplus \dbZ_{>0}\beta) } X_\gamma
\]
for any $\alpha,\beta \in \Phi$ such that
$\alpha \neq -\lambda \beta$ with $\lambda \in \dbR_{>0}$. Then we say that $G$ is graded by $\Phi$ and $\{X_\alpha \}$ is a $\Phi$-grading of $G$.
\end{defn}
 Chevalley groups of adjoint type, the corresponding Steinberg groups (i.e. ``Graded covers" of such Chevalley groups in the sense of \cite{EJK}), and twisted Chevalley groups all fall into this class. For instance:
\begin{eg}
Let $R$ be an (associative) ring. Its Steinberg group at level n, denoted $\St_n(R)$, is the group generated by the symbols
$\{e_{ij}(a)| \ 1 \leq i \neq j \leq n,  a \in R \}$, with relations
\begin{itemize}
\item[(1)]$e_{ij}(a)e_{ij}(b)=e_{ij}(a+b)$
\item[(2)]$[e_{ij}(a),e_{jk}(b)]=e_{ik}(a+b)$ for $i,j,k$ distinct.
\item[(3)]$[e_{ij}(a),e_{jk}(b)]=id$ for $j\neq k, i \neq l$.
\end{itemize}

This group admits an $A_{n-1}$ grading, and 
surjects onto the elementary group
$\mathbb E_{{\it A}_{n-1}}(R)$, a group that is  equal to $\SL_n(R)$ when, for instance, $R$ is an Euclidean domain.
\end{eg}

As a natural generalization of the $SL_n$ case\cite{Ka,Sh1,EJ}, one may wonder when such groups possess property $(T)$. In
the recent work of Ershov, Jakin-Zapirain and Kassabov \cite{EJK}, it was shown that if $G$ is a group graded by $\Phi$, an irreducible classical root system of rank $ \geq 2$, then under a rather
weak restriction on the grading $\Phi$, the union of root subgroups is a Kazhdan subset of $G$. As a result, for such $\Phi$ and a finitely generated commutative ring $R$ , the Steinberg group $St_{\Phi}(R)$ and
the elementary Chevalley group
$\mathbb E_{\Phi}(R)$ have property $(T)$.

The following main theorem of this paper will further these results to groups defined over non-associative rings:
\begin{thm}\label{mainresult}
Let $\mathcal{A}$ be an alternative ring, then
\begin{itemize}
\item[(1)]
 The Steinberg group $\St_3(\mathcal{A})$ is an $A_2$-graded group. It has property $(T)$ when $\biga$ is finitely generated,
\item[(2)]
 If $\biga$ is equipped with a nuclear involution $*$, then the Steinberg group $\St(V)$ (where V is the hermitian Jordan pair $(H_3(\biga,*),H_3(\biga,*))$) is $C_3$-graded. This group has property $(T)$  if $\biga$ is finitely generated and the set of symmetric elements finitely generated as a form ring.

\end{itemize}
\end{thm}

\section{Background}
\subsection{Property $(T)$}
\label{propt}
Being a well known notion in representation theory, equivalent definitions of property $(T)$ abound in literature. We refer the readers to \cite{BHV} for details. This paper will use a version extracted from \cite{EJK}, that suits our setting of \textcolor{red}{discrete} groups:

\begin{defn}\rm Let $G$ be a (discrete) group and $S$ a subset of $G$, 
Let $V$ be a unitary representation of $G$.
A nonzero vector $v\in V$ is called $(S,\eps)$-invariant if
$$\|sv-v\|\leq \epsilon\|v\| \mbox{ for any } s\in S.$$
If there exists a finite set $S \subset G$ and $\epsilon > 0$ such that for any unitary representation $V$ of $G$, the existence of a $(S,\eps)$-invariant vector implies the existence of a vector $v \in V$ fixed by $G$, we say that $G$ {\it has  property ($T$)} or alternatively,  {\it is Kazhdan}. Such $S$ is called a finite Kazhdan subset.
\end{defn}
%
%
%
%
%
%
%
In many circumstances, it is easier to find an infinite Kazhdan set in our group of interest $G$. In these situations, we shall need relative property $(T)$ to establish that the group is Kazhdan.
This notion was originally defined for pairs $(G,H)$
where $H$ is a normal subgroup of $G$ (see \cite{marg,Co}):

\vskip .1cm

\begin{defn}\rm Let $G$ be a group and $H$ a normal subgroup of $G$.
The pair $(G,H)$ has {\it relative property $(T)$} if there exists a finite set $S$
and $\eps>0$ such that if $V$ is any unitary representation of $G$
with an $(S,\eps)$-invariant vector, then $V$ has a (nonzero) $H$-invariant vector.
The supremum of the set of all such $\eps$'s with this property (for a fixed set $S$) is called
the {\it relative Kazhdan constant} of $(G,H)$ with respect to $S$, denoted
$\kappa(G,H;S)$.
\end{defn}

The notion above can be generalized to pairs $(G,B)$, (see \cite{Co}, also a remark in \cite[Section~2]{EJ} )
where $B$ is an arbitrary subset of a group $G$:

\begin{defn}\rm Let $G$ be a group and $B$ a subset of $G$.
The pair $(G,B)$ has {\it relative property $(T)$} if for any
$\epsilon>0$ there is a finite subset $S$ of $G$  and $\mu>0$
such that if $V$ is any unitary representation of $G$ and $v\in V$ is
$(S,\mu)$-invariant, then $v$ is $(B,\eps)$-invariant.
\end{defn}

In practice, the following "bounded generation principle" provides a nice way to generate from given subsets of $G$ with relative property ($T$) bigger ones with such property:

\begin{lem}\label{bddgen}
Let $G$ be a group, $B_1,...,B_k$ a finite collection of subsets of $G$. Suppose that ($G,B_i$) has relative $(T)$ for each i, then ($G,B_1 ...B_k$) also has relative (T). Here $B_1 ...B_k$ is the set of all elements of $G$ representable as $b_1 . . . b_k$ with $b_i \in B_i$ .
\end{lem}

%

\subsection{Case $A_2$ and Alternative rings}\label{a2}

Faulkner\cite{Faul} showed that the notion of a `` Steinberg group of type $A_2$''
could be defined over any alternative ring. We now elaborate on this example for narrational completeness.

\begin{defn}\label{defalt}
A ring $\mcA$ is said to be {\it alternative} if $x^2y=x(xy)$ and $yx^2=(yx)x$ for all $x, y$ in $\mcA$. Such naming is due to the following fact: set $(x,y,z):=(xy)z-x(yz)$, then in all such $\biga$, $(x_{\sigma(1)}, x_{\sigma(2)}, x_{\sigma(3)}) = \sgn(\sigma)(x_1,x_2,x_3)$ is true for any $x_i \in \biga$ and any $\sigma \in S_3$.
\end{defn}

\begin{egs}
Any associative ring; the octonions $\dbO$.
\end{egs}

\begin{defn}
For $R$ associative, we can define the following:
\begin{itemize}
\item (``The elementary group"):
$E_n(R)$ is the  group generated by the elementary matrices $I+aE_{ij}$ where $a \in R, \ i \neq j$.
Denote its center 
by $Z_0$. It has a projective version $PE_n(R):= E_n(R)/Z_0$
\item(``The elementary Lie algebra"):
 $e_n(R)$  is the Lie ring generated by
$aE_{ij}$ , where $ a \in R,  \ i \neq j$. Denote its center by $\mathcal{Z}_0$.
It also has a projective version $pe_n(R):= e_n(R)/ \mathcal{Z}_0$.
\end{itemize}
Over a \textcolor{red}{non-associative ring} $R$, assuming $n \geq 3$, we can  mimic the above construction:
\begin{itemize}
\item(``The Steinberg group'')
$St_n(R)$ is the group generated by the symbols $x_{ij}(a),1 \leq i\neq j
\leq n \, a \in R$ subject to the relations:
\begin{equation*}
\begin{array}{ccccc}

(1) & x_{ij}(a)x_{ij}(b) &= & x_{ij}(a+b),\\

(2) &[x_{ij}(a),x_{ik}(b)] &= & x_{ik}(ab) & \text{for $i,j,k$ distinct},\\

(3) &[x_{ij}(a),x_{ik}(b)] &= & 1 & \text{for $j \neq k$ and $i \neq l$}. \\

\end{array}
\end{equation*}

\item(``The Steinberg Lie algebra'')
$st_n(R)$ is the Lie algebra generated by the symbols $l_{ij}(a),1 \leq i\neq j
\leq 3 , \ a \in R$ subject to the relations:
\begin{equation*}
\begin{array}{ccccc}

(4) &l_{ij}(a)+l_{ij}(b) &= & l_{ij}(a+b),\\

(5) &[l_{ij}(a),l_{ik}(b)] &= & l_{ik}(ab) & \text{for $i,j,k$ distinct},\\

(6) &[l_{ij}(a),l_{ik}(b)] &= & 0 & \text{for $j \neq k$ and $i \neq l$.} \\
\end{array}
\end{equation*}

\end{itemize}
\end{defn}

\begin{rem}
Of course one may see "collapsing" resulting from this construction, i.e. the relations might yield groups/Lie algebras with root
subgroups/subspaces smaller than $R$. However, this collapsing can be avoided: it was shown in \cite{Faul}
that if $n \geq 4$ and $R$ is associative, or if $n=3$ and $R$ is \textcolor{red}{alternative}, then there is no collapsing.
\end{rem}

The adjoint action of $E_n(R)$ on $e_n(R)$ is defined as:
\[
(I+aE_{ij})x(I-aE_{ij})=x+ [aE_{ij},x]- (aE_{ij})x(aE_{ij})
\]
Using the fact that both $Z_0$ and $\mathcal{Z}_0$ lie inside $Z(R)I$,it follows that $PE_n(R)$ acts faithfully on $pe_n(R)$. By expanding the second and third term of the right hand side, Faulkner \cite{Faul} gave a definition of $PE_3(R)$ for $R$ an alternative ring :

\begin{defn}\label{a2exponential}
If $a \in R$ is an alternative ring, define the following map on $pe_3(R)$: $e_{ij}(a)=id + d_1+ d_2$, where $d_1(x)=[aE_{ij},x]$ and $d_2(x)=-(aba)E_{ij}$ if the $j,i$-component
of x is $bE_{ji}$. Since $(ab)a=a(ba)$ in an alternative ring, we are safe to omit the parentheses for $aba$.
\end{defn}

\begin{prop}[\cite{Faul},(A.7)] \label{altA2gddgp} If $R$ is an alternative ring, then $e_{ij}(a)$ is an automorphism of $pe_3(R)$. Moreover,
\begin{itemize}
\item[(a)]$e_{ij}(a)e_{ij}(b)=e_{ij}(a+b)$
\item[(b)]$[e_{ij}(a),e_{jk}(b)]=e_{ik}(a+b)$ for $i,j,k$ distinct.
\item[(c)]$[e_{ij}(a),e_{jk}(b)]=id$ for $j\neq k, i \neq l$.
\end{itemize}

The group generated by $\{e_{ij}(a)| \ 1 \leq i \neq j \leq 3, a \in \biga \}$ will be denoted by ${\rm PE}_3(\biga)$.
\end{prop}

\begin{proof}
See \cite[A.7]{Faul}.
\end{proof}

The following is an analog of Steinberg groups over associative rings:
\begin{defn}
For an alternative ring $\mathcal{A}$, denote by $ \St_3(\biga)$ the group generated by the the set of symbols $\{e_{ij}(a) | 1 \leq i\neq j \leq 3 , a \in \biga \}$ satisfying the relations (1)-(3) of Proposition \ref{altA2gddgp}.
\end{defn}

The commutator relations above endow $A_2$-gradings to both $\text{PE}_3(\mathcal{A
})$ and $\St(\mathcal{A})$. Under the standard realization $\{w_i-w_j | 1 \leq i \neq j \leq 3\}$ of $A_2$, their root subgroups are $G_{w_i-w_j}=G_{ij}:=\la e_{ij}(a),  a \in \biga \ra$. This shall be a starting point for our later considerations.


\subsection{Jordan Theory}
The above construction is actually a special case of {\it groups arising from Jordan pairs} that we will present in this section. Nice references are \cite[Section \S7-\S10]{ln} (or the short version \cite{lns}) and \cite{loos}, from which most notations in this section are borrowed.

\begin{defn}\label{quadmap}
Let $k$ be a commutative ring. A {\it quadratic map} between $k$-modules $Q: M \rightarrow N$ is one such that the following holds: \begin{itemize}
\item[(1)] For all $a \in k,\ v \in M$, $Q(av)=a^2Q(v)$,
\item[(2)] For $x,y \in M$, $Q(x+y)-Q(x)-Q(y)$ is a $k$-bilinear map.
\end{itemize}
\end{defn}
\begin{defn}[\cite{lns}]
Let $k$ be as in above, $\sigma \in \{\pm \}$.
Then a {\it{Jordan pair} }$V=(V^+, V^-)$ is a pair of $k$-modules together with a pair $(Q_+, Q_-)$ of quadratic maps $Q_\sigma: V^\sigma \rightarrow \Hom_k(V^{-\sigma}, V^{\sigma})$ such that the following identities hold in all base ring extensions of $V$:
\begin{align}
D_\sigma(x,y)Q_\sigma(x) &= Q_\sigma(x)D_{-\sigma}(y,x),\\
D_\sigma(Q_\sigma(x)y,y) &= D_\sigma(x,Q_{-\sigma}(y)x), \label{1.2}\\
Q_\sigma(Q_\sigma(x)y) &= Q_\sigma(x)Q_{-\sigma}(y)Q_\sigma(x),
\end{align}
Here $D_\sigma : V^\sigma \times V^{-\sigma} \rightarrow {\rm End}_k(V^\sigma)$, is 
the bilinear map associated to $Q$: 
\[
D_\sigma(x,y)z = Q_\sigma(x+z)y - Q_\sigma(x)y - Q_\sigma(z)y,
\]
In the succeeding text, we may also use the triple product notation $\{x,y,z\}:=D_\sigma(x,y)z$ where $x,z \in V^\sigma$, $y \in V^{-\sigma}$ . Following \cite[chapter 2.0]{loos}, the subscript $\sigma$ may be dropped
when clearly implied by the context, and we may write $Q_x y$ for $Q(x)y$.
\end{defn}

\begin{rem}
Let $k$ be as in above.
A {\it{(quadratic) Jordan algebra} }$J$ is a $k$-module on which a quadratic map $U_x: J \rightarrow \End_k(J)$ is defined, such that the following operator identities hold in all base ring extensions of $J$:
\begin{align}
V_{x,y}U_x &= U_xV_{y,x},\\
V_{U_x y,y} &= V_{x,U_y x},\\
U_{U_x y} &= U_xU_yU_x,
\end{align}
Here $V : J \times J \rightarrow {\rm End}_k(J)$ is 
the bilinear map defined by
\[
V_{x,y}(z) = U_{x+z}(y) - U_{x}(y) - U_{z}(y),
\] denoted $\{x,y,z\}$. This operator $J$ is called unital if there exists an element $1 \in J$ such that $U_1 = \text{Id}_J$.
\end{rem}

\begin{rem}\label{j&jp}
The concept of Jordan pairs is a natural generalization of (Quadratic) Jordan algebras: given any Jordan algebra $J$, the pair $(J,J)$ with $Q_x y:=U_x(y)$, where $x,y \in J$ and $U_x$ the quadratic operator in $J$, is automatically a Jordan pair since $Q_x$ and $U_x$ satisfies exactly the same defining identities. So in this setting, we will use $Q_x$ and $U_x$ interchangeably. Similarly with $V_{x,y}$ and $D(x,y)$.
\end{rem}
\begin{egs}
Among other constructions, a major source of Jordan pairs is matrices. E.g. Let $k$ be a field. if
\begin{itemize}
\item[($\alpha$)]$V^+ = V^-= \mathbb{S}_n(k)$, the symmetric $n\times n$ matrices over $k$, or
\item[($\beta$)]$V^+ = M_{pq}(R) , V^-= M_{qp}(R)$ where $M_{ij}$ are $i \times j$ matrices over an $k$-algebra $R$, or
\item[($\gamma$)]$V^+ = V^-= \mathbb{A}_n(k)$, the skew-symmetric $n\times n$ matrices over $k$ with zero diagonal,

\end{itemize}
Then $(V^+, V^-)$ is a Jordan pair where the quadratic map $Q_x y=xyx$ is just the matrix product.
\end{egs}
A $homomorphism \ h:V \rightarrow W $ {\it of Jordan pairs} is a pair of linear maps $(h_+, h_-)$ that preserves the Jordan pair structure, namely $h_\sigma: V^\sigma \rightarrow W^\sigma$ satisfies $h_\sigma(Q_x y)= Q_{h_\sigma(x)}h_{-\sigma}(y)$ for all  $x \in V^\sigma, y \in V^{-\sigma}$. The definitions of endomorphism, isomorphism and automorphism of Jordan pairs clearly follow.

\begin{rem}
If a Jordan pair can be embedded into $V=(M_{pq}(R),M_{qp}(R))$ where $R$ is an associative ring, we call it {\it special}. Elsewise we call it {\it exceptional}.
\end{rem}

\begin{egs}
For a pair $(x,y) \in V^\sigma \times V^{-\sigma}$, the Bergmann operator $B(x,y)\in \text{End}V^\sigma$ is
\[
B(x,y)= \text{Id}_{V^\sigma} - D(x,y) + Q_xQ_y
\]
If $B(y,x) \in \text{End}V^{-\sigma}$ is invertible, then $\beta(x,y):=(B(x,y) , B(y,x)^{-1})$ is an isomorphism of $V$. It is easily seen that the same concept is defined for a single Jordan algebra.
\end{egs}

The notion of a derivation is defined as follows: a pair $\Delta = (\Delta_+, \Delta_-)$ of linear maps $\Delta_\sigma \in {\rm End}(V^\sigma)$ is called a $derivation$ if Id+$\epsilon\Delta$ is an automorphism of the base ring extension $V \otimes_k k(\epsilon)$, where $k(\epsilon)$ is the ring of dual numbers. This is equivalent to the validity of the formula:
\[
\Delta_\sigma(Q_z v) = \{\Delta_\sigma(z),v,z\}+Q_z \Delta_{-\sigma}(v)
\] for all $z \in V^\sigma$, $v \in V^{-\sigma}$. As with derivations of other algebraic structures, ${\rm Der}(V)$ forms a Lie algebra under component-wise operations.

Among all derivations, we need only the following two: the inner derivations $\delta(x,y):= (D(x,y), -D(y,x))$, and the central derivation $\zeta_V = ({\rm id}_{V^+} , {\rm id}_{V^-})$. Invoking one of the main Jordan identities, it could be checked that ${\rm Inder}(V):= {\rm span} \{\delta(x,y) \mid (x,y) \in V \}$ is a Lie subalgebra of ${\rm Der}(V)$.

Setting
$\mathfrak{L}_0 = k \cdot \zeta_V + {\rm Inder}(V)$,
the celebrated procedure of Tits-Kantor-Koecher allows us to construct  from any Jordan pair $V$ the following Lie algebra
\[
\mathcal{TKK}(V)= V^+ \oplus \mathfrak{L}_0(V) \oplus V^-.
\]
with multiplication
\setcounter{equation}{7}
\[
[V^\sigma, V^{-\sigma}]=0, \ \quad
 [D,z]=D_\sigma(z), \ \quad
  [x,y]=-\delta(x,y) \tag{\theequation}
\]
for $D= (D_+, D_-) \in \mathfrak{L}_0 (V), \ z \in V^{\pm}$ and $(x,y) \in V$. It follows immediately that if $ x \in
V^\sigma , y \in V^{-\sigma}$, then \[2Q_x y=\{xyx\}=[x,[x,y]] \ \ . \setcounter{equation}{8}\tag{\theequation}\label{quadandbracket}\] This suggests that  $\frac{1}{2}ad_x^2$ is meaningful even when $\frac{1}{2}$ is 
non-existent in the base ring. Along with the fact that $ad_x^3=0$, we are enabled to make a quadratic definition of ``exponentiation'' as follows:

\vskip .2cm
For any $x \in V^\sigma (\sigma \in \{\pm \})$, ${\rm exp}_\sigma(x)$ is an endomorphism of $\mathcal{TKK(V)}$  given by:
\[ {\rm exp}_\sigma (x)z =x,\ \quad
{\rm exp}_\sigma (x)\Delta = \Delta + [x,\Delta]  ,\ \quad
{\rm exp}_\sigma (x)y = y + [x,y] +Q_x y,
\]
where $z \in V^\sigma, \ \Delta \in \mathfrak{L}_0$ and $y \in V^{-\sigma}$. It could be checked that $\exp_\sigma(z)$ is an automorphism of $\mathcal{TKK}(V)$ and
$\exp_\sigma : V^\sigma \rightarrow {\rm Aut} (\mathcal{TKK}(V))$ is actually an injective homomorphism \cite{loos2}.

Putting $U^\pm := {\rm Im(exp_\pm)}$, the {\it projective elementary group of V} is the subgroup of
${\rm Aut(\mathcal{TKK}}(V))$ generated by $U^+$ and $U^-$. Such a group will be denoted $PE(V)$.

\begin{egs}
Let's assume $\biga$ is not associative. Then ${\rm PE}_3(\biga)$, as described in \S\ref{a2}, is actually a projective elementary group with respect to the exceptional Jordan pair $V=(M_{12}(\biga), M_{21}(\biga))$ (See \cite{ln} for definitions). We will encounter another exceptional Jordan pair in $\S$ 3.
\end{egs}

\section{Establishing Relative Property $(T)$}

The main references for this section are \cite{Sh1,Ka}

As mentioned in section \ref{propt}, a typical way to prove that a group $G$ has property $(T)$
is to find a subset $K$ of such that
\begin{itemize}
\item[(a)] $K$ is a Kazhdan subset of $G$
\item[(b)] the pair $(G,K)$ has relative property $(T)$.
\end{itemize}

\begin{rem}
Quoting \cite{EJK}:

\vskip .1cm
\it {Clearly, (a) and (b) imply that $G$ has property $(T)$. Note
that (a) is easy to establish when $K$ is a large subset of $G$,
while (b) is easy to establish when $K$ is small, so to obtain
(a) and (b) simultaneously one typically needs to pick $K$
of intermediate size.}
\end{rem}
%


We're now ready to present the principal  theorem of this section.
\begin{thm}
\label{KasShalomZ}
Let $\biga$ be a finitely generated alternative ring, $\mathcal{A}\ast \mathcal{A}$ the free product of two
copies of the additive group of $\mathcal{A}$, and consider the semi-direct
product $(\mathcal{A}\ast \mathcal{A})\ltimes \mathcal{A}^2$, where the first copy of $\mathcal{A}$
acts by upper-unitriangular matrices, that is, $a\in \mathcal{A}$ acts
as left multiplication by the matrix $e_{1,2}(a)=$
$\left(
\begin{array}{cc}
1& a\\ 0& 1
\end{array}
\right)$,
and the second copy of $\mathcal{A}$
acts by lower-unitriangular matrices. Then the pair
$((\mathcal{A}\ast \mathcal{A})\ltimes \mathcal{A}^2,\mathcal{A}^2)$ has relative property $(T)$.
\end{thm}

\vskip .2cm
It is obvious that we only have to prove the theorem for $\biga_n$, the free alternative rings generated by $n$ elements. So  we shall use from now on $\biga_n$ and $\biga$ interchangeably.  But before we proceed, let's first analyze the situation.

Denote by $L(\biga)$ the left multiplication ring of $\biga$, defined as the ring generated by the left multiplication operators $L_a(a \in \biga)$ where $L_a(b)=ab$. By the way the action is defined, one can replace the group $\biga \ast \biga$ by a subgroup of $L(\biga) \ast L(\biga)$. Formally speaking, this is given by the following map:

\begin{lem}\label{laisa}
There exists a monomorphsm of abelian groups $g:\biga \rightarrow L(\biga)$ given by $a \mapsto L_a$.
\begin{proof}
Since $1 \in \biga$, it is obvious that the above map is injective.
\end{proof}
\end{lem}

\noindent Thus the replacement is justified.

%

Subtleties arise in this nonassociative situation of ours. However, they can be circumvented if we utilize following observations (whose proof will be omitted).


\begin{rem}\label{important}
\hfill
\begin{itemize}
\item[($\beta$)]
\cite[pg8,13,19]{ivan} The free alternative ring $\dbZ_\mathfrak{A} \la X \ra$ is a graded ring, namely
\[
\dbZ_\mathfrak{A} \la X \ra = \bigoplus\limits_{i=0}^\infty \la X \ra ^i
\]
where $\la X \ra$ is the ideal generated by the set $X$.

\item[($\gamma$)]\cite{zhev}\cite[pg122 .3,.4]{ivan} Let $\biga$ be an alternative ring with a set of generators $\{a_i\}_{i \in I}$. Then the left multiplication ring $L(\biga)$
is generated by
$S_{L(\biga)}:= \{L_u \mid u=(a_{i_k}(a_{i_{k-1}}(\cdots (a_{i_3}(a_{i_2}a_{i_1})))\ldots)) ,i_1<i_2< \ldots<i_k \}$.
Following \cite{ivan}, we call the set of $u$'s involved in $S_{L(\biga)}$ the set of $r_1$-words. Since finite generation of $\biga$ implies the finiteness of the corresponding set of $r_1$-words, $L(\biga)$ is finitely generated under this assumption. It is clear that $L(\biga)$ is associative.

For brevity we use the notation $\mathscr{S} := S_{L(\biga)} \cup \{Id_\biga \}$.
\end{itemize}
\end{rem}
\vskip .1 cm

Instead of proving the precise statement of Theorem \ref{KasShalomZ}, we'll handle a slightly stronger version that is
a direct consequence of

\begin{thm*}\cite[Theorem 1.2]{Ka}

let $R$ be any finitely generated associative ring, $R \ast R$ the free product of two
copies of the additive group of $R$, and consider the semi-direct
product $(R\ast R)\ltimes R^2$, where the first copy of $R$
acts by upper-unitriangular matrices, that is, $a\in R$ acts
as left multiplication by the matrix $e_{1,2}(a)=$
$\left(
\begin{array}{cc}
1& a\\ 0& 1
\end{array}
\right)$,
and the second copy of $R$
acts by lower-unitriangular matrices. Then the pair
$((R\ast R)\ltimes R^2,R^2)$ has relative property $(T)$.
\end{thm*}

Or rather, of what the author's argument actually proved:

\begin{thmstar*}
Let the set $S$ be the finite generating set of $R$ adjoining the unit element. Then the pair
 $(( \dbZ S  \ast \dbZ S )\ltimes R^2,R^2)$ has relative property $(T)$. Here $\dbZ S $ is the additive subgroup of $R$ generated by $S$.
\end{thmstar*}

We are now ready to prove Theorem \ref{KasShalomZ} as an implication of the following:

\begin{prop}
Let $\mathscr{A}$ be an abelian group; $\varPhi:=\{\varphi_1, \ldots \varphi_n\} \cup \{Id_{ \mathscr{A}} \} \subset \End_{\mathbb{Z}}(\mathscr{A})$. 
Assume that there exists an element $a \in \mathscr{A}$ such that $\mathscr{A} =\la \varPhi \ra a$, where $\la \varPhi \ra$ is the ring generated by $\varPhi$. If we define the action as in Theorem \ref{KasShalomZ}, then
$(( \dbZ \varPhi  \ast \dbZ \varPhi )\ltimes \mathscr{A}^2,\mathscr{A}^2)$ has relative property $(T)$.
\end{prop}

\begin{proof}
Set $R:=\dbZ \la x_1 ,\ldots , x_n \ra, \ S:= \{x_1, \ldots ,x_n\}$. There exists a map $f$ between the pairs $(( \dbZ S  \ast \dbZ S )\ltimes R^2,R^2)$ and $(( \dbZ \varPhi  \ast \dbZ \varPhi )\ltimes \mathscr{A}^2,\mathscr{A}^2)$, given by the data $f|_{\dbZ S}: x_i \mapsto \varphi_i ;  1 \mapsto Id_\mathscr{A} $
and
$f|_{R^2}: (P_1(x_1, \ldots ,x_n) , P_2(x_1, \ldots ,x_n)) \rightarrow (P_1(\varphi_1, \ldots ,\varphi_n)a , P_2(\varphi_1, \ldots ,\varphi_n)a) ; \  1 \mapsto a$. It is clear that this map extends to an epimorphism of groups. 
\end{proof}

\noindent{\it Proof of Theorem \ref{KasShalomZ}}
\ \ Define $g$ as in Lemma \ref{laisa}. Recall that $\mathscr{S} := S_{L(\biga)} \cup \{Id_\biga \}$. Setting  $\varPhi$ to be $\mathscr{S}$, $\mathscr{A}$ to be $\biga$, $a=1$ in the previous proposition, we get that $(( \dbZ \mathscr{S}  \ast \dbZ \mathscr{S} )\ltimes \biga^2,\biga^2) \cong (( g^{-1}(\dbZ \mathscr{S})  \ast g^{-1}(\dbZ \mathscr{S}) )\ltimes \biga^2,\biga^2)$ has relative property $(T)$, which implies Theorem \ref{KasShalomZ} since the latter embeds into $((\mathcal{A}\ast \mathcal{A})\ltimes \mathcal{A}^2,\mathcal{A}^2)$, preserving $\biga^2$.

\begin{rem}\hfill
\begin{itemize}
\item[(1)]The above proof could be extended effortlessly to any non-associative unital ring $R$, as long as the left multiplication ring $L(R)$ is finitely generated .
\item[(2)]The theorem works just as well if the action is given instead by right multiplication.
\end{itemize}
\end{rem}

Combining the above result with \cite[Theorem 5.1]{EJK},
We can now prove the first part of Theorem \ref{mainresult}:

\begin{thm}
For $\mathcal{A}$ a finitely generated alternative ring, $\St_3(\mathcal{A})$ is Kazhdan.
\end{thm}
\begin{proof}
It is obvious that $\St_3(\biga)$ has a strong $A_2$ grading in the sense of \cite[\S 4.4]{EJK}, which roughly means that if $\gamma \in A_2$ is a positive non simple root, then $G_\gamma$, the corresponding root subgroup in $\St_3(\biga)$ is contained in the subgroup generated by the simple-root subgroups. This implies that \cite[Theorem 5.1]{EJK} applies, giving $\kappa(\St_3(\biga),\bigcup G_{ij}) >0$

To finish the proof, it is necessary to establish $\kappa_r(\St_3(\biga),G_{ij}; S) >0$ for some  finite set $S \subset \St_3(\biga)$. Every root subgroup $G_{ij}$ of $\St_3(\biga)$ belongs to some subgroup $N_{ij}$ isomorphic to $\biga^2$, which is in turn contained naturally in $T_{ij}$, a homomorphic image of $(\biga \ast \biga)\ltimes \biga^2$ embedded inside $\St_3(\biga)$ (See \cite[Lemma 2.4]{Sh1})(For example,
we can set $T_{12}= \la G_{23},G_{32}\ra \ltimes \la G_{12}, G_{13} \ra$). Therefore Theorem \ref{KasShalomZ} guarantees that for every $G_{ij}$ there exists a finite subset $S_{ij} \subset T_{ij}$ such that
\[
\delta_{ij}:=\kappa_r(\St_3(\biga),G_{ij}; \bigcup\limits_{i \neq j} S_{ij})
\geq
\kappa_r(\St_3(\biga),N_{ij}; \bigcup\limits_{i\neq j} S_{ij})
\geq
\kappa_r(T_{ij},N_{ij}; S_{ij})
>
0 .
\]
Since
$\kappa_r(\St_3(\biga),\bigcup G_{ij}; \bigcup S_{ij}) \geq \inf\limits_{i \neq j}\delta_{ij} > 0$ , the Theorem is proved.
\end{proof}

It is known that Property $(T)$ is inherited by homomorphic images. Consequently we have

\begin{cor}
$\text{PE}_3(\biga)$ is Kazhdan.
\end{cor}

\section{A Symplectic {Group over Alternative Rings}}

%

In this section , $\biga$ will always be a finitely generated alternative ring with a nuclear involution $*$, namely every $*$-symmetric element $\alpha$ satisfies the identity $(\alpha,a,b) := (\alpha a)b- \alpha(ab)=0$ for $ a,b \in \biga$. The set of
$\ast$-symmetric elements will be denoted as $Sym(\biga)$. We will also use the standard realization of $C_3: \{\pm w_i \pm w_j\} \cup \{\pm 2w_i\}$. For the Jordan identities involved, see \cite[\S7.6]{ln}

\vskip .4cm
Aside from the the aforementioned $A_2$ case, there are more instances of low rank ``root graded groups'' that are defined over alternative rings. We'll describe such an example below. It will be referred to as a "symplectic group" since such group comes with a $C_3$-grading.

    More precisely, the group we are to consider is
${\rm PE}(V)$, where $V$ is the Jordan pair $(H_3(\biga,*),H_3(\biga,*))$ (which is exceptional when $\biga$ is not associative). Here $H_3(\biga,*)$ (usually  called Hermitian matrices) is the linear subspace of $M_3(\biga)$ consisting of symmetric elements under the involution map $J_1: (x_{ij}) \rightarrow (x_{ji}^*)$ of the full matrix space. To avoid confusion, we may use superscripts $\pm$ to distinguish elements from the two components.

As in Remark \ref{j&jp}, the Jordan pair structure on $V$ is dictated by the Jordan algebra structure of $H_3(\biga,*)$,
which is well understood. We will refer to \cite[Pg 1075]{mc} for a fully worked out (quadratic) Jordan multiplication table with respect to the following basis of $\mathcal{H}_3(\biga, \ast)$, given in terms of matrix units $E_{ij} (1\leq i,j\leq 3)$:
\begin{align*}
\alpha[ii]&:=\alpha E_{ii} \ , \ \alpha=\alpha^* \in \biga \\
d[ij]&:= d E_{ij} + d^* E_{ji} \ , \ d \in \biga \\
[ij]&:= E_{ij} + E_{ji} \ , [ii]:= E_{ii}
\end{align*}

Note that under this notation, $d[ij]:= d E_{ij} + d^* E_{ji}=d^*[ji]$.

\begin{rem}
Denote by $xy$ the usual matrix product in $\mathcal{H}_3(\biga, \ast)$, and the brace product $\{x,y\} := xy+yx$. Then
the expression $U_x(y)=\frac{1}{2}\{x,\{x,y\}\} - \{x^2,y\}$ is well defined, regardless of the presence (or lack thereof) of $\frac{1}{2}$ in
$\mathcal{H}_3(\biga, \ast)$.
\end{rem}

\begin{rem}\label{multh3}
(\cite[Pg 1075]{mc}) To determine the products in $\mathcal{H}_3(\biga, \ast)$,  one has to go no further than computing them for the basis elements. They are as follows:

Quadratic products ($U_x$):
\begin{align*}   U_{\alpha[ii]}{\beta[ii]}&=\alpha\beta\alpha[ii], \\  U_{d[ij]}{b[ij]}&=db^*d[ij] \ (i \neq j) ,  \\
U_{d[ij]}{\beta[jj]}&=dbd^*[ij] \ (i \neq j) \  ,\\ 
U_{a[ij]}{b[kl]}&=0 \ \ {\it if} \ \{k,l\} \neq \{i,j\}
\end{align*}

Triple products ($\{x,y,z\}$):
\begin{align*}  
\{d[ij],b[jk],c[kl]\}&=\frac{1}{2}(d(bc)+(db)c)[il], \\  
\{d[ij],b[ji],c[ik]\}&=d(bc)[ik], \\  
\{d[ij],b[jk],c[ki]\}&=(d(bc)+(bc)^*d^*)[ii], \\  
\{d[ij],b[ji],c[ii]\}&=(d(bc)+(bc)^*d^*)[ii].
\end{align*}
Also, a triple product on basis elements can be non zero if and only if  it can be written in the form  $\{a[ij],b[kl],c[pq]\}$, where $j=k$ and $p=q$. 
\end{rem}

Let's compare $\fkg=\mathcal{TKK}(V)$ to the usual symplectic Lie algebra $\mathfrak{c}=\mathfrak{sp}_6(F)$, where $F$ is an algebraically closed field:

In matrix terms, the condition for
$x= \left(
\begin{array}{cc}
A& B\\ C& D
\end{array}
\right) \in \mathfrak{c}$
$(A,B,C,D \in \mathfrak{gl}(3,F))$ to be symplectic is that $B^T=B, C^T=C,$ and $A^T= -D$. Under this realization, entries of $B$ (resp. $C$) correspond to the root subspaces of $\{w_i+w_j\} \cup \{2w_i\}$(resp. $\{-w_i-w_j\} \cup \{-2w_i\}$), while
$\left(
\begin{array}{cc}
A& 0\\0& D
\end{array}
\right)$
contains the root subspaces of $\{w_i-w_j\}$ and the Cartan subalgebra.

 This root decomposition actually carries over to $\mathcal{TKK}(V)$(see \cite{Jac}) when we treat $V^\pm$ like the upper left and lower right block matrices above. Explicitly, the root subspaces are:
\vskip .1cm
\begin{center}
(Table 1) \ \ 
    \begin{tabular}{| l | l |}
    \hline \\[-2.5ex]

    root $\gamma$ & root space $W_\gamma$  \\ \hline \\[-2.5ex]
     $\pm 2w_i$ & $\alpha[ii]^{\pm} \ (　\ \alpha=\alpha^* \in \biga)$
     \\ \hline \\[-2.5ex]
     $\pm (w_i + w_j) $ & $d[ij]^\pm \ ( \ d \in \biga)$
     \\ \hline \\[-2.5ex]
     $ (w_i - w_j)$ & $ \{\delta(x.y) |x \in W_\mu, y \in W_\rho, \mu \ {\rm positive}, \rho \ {\rm negative}, \mu + \rho = w_i-w_j\}$
     \\ \hline
    \end{tabular}\\
\end{center}
In particular, this tells us that $V^\pm$ are direct sums of the root subspaces with respect to $\pm \{ \{2w_i\} \cup \{w_i + w_j\} \} $.

\vskip .2cm
When a Lie algebra has root space decomposition (with respect to a root system $\Phi$), one naturally expects its corresponding ``Chevalley group of adjoint type" (see \cite{Ca}), obtained from exponentiation, to obey $\Phi$-commutator relations. This is indeed the case for our ${\rm PE}(V)$, as suggested by

\begin{thm}\label{c3ln}\cite[Theorem 2]{lns}
For the Jordan pair $V$ defined above, $G={\rm PE}(V)$ has a $C_3$-grading with the following root subgroups:
\begin{fleqn}[0pt]
\begin{flalign*}
 \qquad G_\alpha =  \begin{cases}
 \exp_\pm(W_\alpha), \ {\text for} \ \alpha \in \pm \{\{2w_i\}\cup\{w_i+w_j\} \}, \\
 \la \bigcup  \{\beta(W_\mu,W_\rho) | \mu \ {\rm positive}, \rho \ {\rm negative}, \mu + \rho =\beta\} \ra, \ {\text for} \ \alpha \in  \{w_i-w_j\}
 \end{cases}
\end{flalign*}
\end{fleqn}
\end{thm}
Although the theorem does reveal all the root subgroups, it alone is not sufficient for our purpose, requiring us to compute some commutator relations for $\text{PE}(V)$.

We first present some useful notations from \cite{ln}. For $(x,y) \in V$, since $\exp_+(x)$ and $\exp_-(y)$ are automorphisms of $\fkg=\mathcal{TKK}(V)= V^+ \oplus \mathfrak{L}_0(V) \oplus V^-$, we can adopt the following matrix notation :
$$
\exp_+(x)=\left(
\begin{array}{ccc}
1 & ad_x & Q_x\\
0 & 1 & ad_x\\
0 & 0 & 1\\
\end{array}
\right)
, \quad
\exp_-(y)=\left(
\begin{array}{ccc}
1 & 0 & 0\\
ad_y & 1 & 0\\
Q_y & ad_y & 1\\
\end{array}
\right)
$$.

In the same spirit, if $h=(h^+, h^-)$ is an automorphism of $V$, then using the matrix notation one can write the induced automorphism on $\fkg$ as:
\[\left(
\begin{array}{ccc}
h^+ & 0 & 0\\
0 & h_0 & 0\\
0 & 0 & h^-\\
\end{array}
\right)
\]
where $h_0(-\delta(x,y)) = h_0([x,y]) = [h_+(x), h_-(y)]$.

\vskip .3cm
The goal of obtaining a reasonable form of commutator relations require us to ``coordinatize'' each root subgroup. This is easy for $\pm \{\{2w_i\}\cup\{w_i+w_j\} \}$ as they have canonical coordinates coming from the matrix structure of $\mathcal{H}_3(\biga, \ast)$. In the succeeding paragraphs we resolve the problem for the remaining roots.

\begin{prop}
Let $x \in W_\gamma \subset V^+ ,\ y \in W_\tau \subset V^- , \gamma \neq - \tau , \gamma + \tau \in C_3$, then the Bergmann operators $B(x,y)$ and $B(y,x)$ are invertible.

\end{prop}

\begin{proof}
By \cite[Pg 87]{ln},it suffices to prove invertibility for $B(x,y)$ where $x,y$ are basis elements of $H_3(\biga, \ast)^+$ and $H_3(\biga, \ast)^-$, respectively. We shall need the following
\renewcommand{\qedsymbol}{}
\end{proof}

\begin{lem}
As operators on $V^+$, the following holds:
\begin{enumerate}
\item $ ad_x ad_y = -D(x,y)$
\item $ ad_x ad_y Q_x Q_y = Q_x Q_y ad_x ad_y$
\item  $ (Q_xQ_y)^2=0$
 
\end{enumerate}
\end{lem}

\begin{proof}

Since for $z \in V^+$, $\ad_x\ad_y z=[x,[y,z]]=[[x,y],z]]+\cancel{[y,[x,z]]}$, (1) is proved.

Using the result of (1) and symmetry, we see that (2) is equivalent to the first of the defining identities of Jordan pairs, showing its validity.

The general proof of (3) involves a lengthy calculation using the data in Remark \ref{multh3}, but we can give a shorter proof when $\frac{1}{2} \in \mathcal{A}$ .

Recall that for a Jordan Pair $V$ and the corresponding Lie algebra $\mathcal{TKK}(V)$, we can relate the bracket and the quadratic operator by equation (\ref{quadandbracket}):

\[2Q_x( y)=\{xyx\}=[x,[x,y]]\]
for $ x \in
V^\sigma , y \in V^{-\sigma}$.  This tells us that in our case if $ x \in W_\gamma $ , $y \in W_\tau$, then $Q_x(y) \in W_{2\gamma+\tau} $, $Q_y(x) \in W_{2\tau+\gamma} $. Specifying $\gamma$ and $\tau$ to be in $\pm\{\{ w_i +w_j\} \cup \{ 2w_i\}\} \subset C_3$, it is clear that one of the two expressions $2\gamma+\tau$, $\gamma+2\tau$ is not a root, implying either $Q_x(y)$ or $Q_y(x)$ is zero. Now (3) follows from the third defining identity of Jordan pairs. \end{proof} 

And one general identity that holds for any Jordan pair:
\begin{equation}\label{id for Dxysquared}
\{\{xyu\}, y, z\} = \{x, Q_y u, z\} + \{u, Q_y x, z\}.
\end{equation}, which is a direct consequence of identity ({\ref{1.2}})
\begin{claim}
Under the above assumptions, $B(x,y)^{-1}=\text{Id}- ad_x ad_y + Q_xQ_y$
\end{claim}

\begin{proof}
Using the lemma above for all admissible $x,y$'s, we have:
\begin{fleqn}[2pt]
\begin{align*}
& \ B(x,y) (\text{Id}- ad_x ad_y + Q_xQ_y)\\
=& \ (\text{Id}+ ad_x ad_y + Q_xQ_y)(\text{Id}- ad_x ad_y + Q_xQ_y)\\
=& \ \text{Id} - (ad_x ad_y)^2  \cancel{+ad_x ad_y Q_x Q_y} \cancel{-Q_x Q_y ad_x ad_y } + {(Q_xQ_y)^2} + 2Q_xQ_y\\
=& \ \text{Id} - (ad_x ad_y)^2 +4Q_xQ_y -2Q_xQ_y\\
&\text{Since}\\
&(- (ad_x ad_y)^2 +4Q_xQ_y -2Q_xQ_y)a \\
=&- \{\{ayx\},y,x\} +\{x,Q_y a, x\} = - \{x,Q_y a, x\} -\{x,Q_y x , a\} +\{x,Q_y a, x\} =-\{x, Q_y x, a\}
\end{align*}
\end{fleqn}
(Note the usage of identity (\ref{id for Dxysquared})), it would suffice to prove  $\{x, Q_y x, a\}=0$ for all $a \in V^+$. This essentially breaks down to 2 cases:
\begin{enumerate}
\item $x=c[ii], y=b[ij]$. Here $Q_y(x)=b^*cb[jj]$, so $\{x, Q_y x, a\}=-\{a[ii], b^*cb[jj], a\}=0$.
\vskip .2cm
\item $x=c[ij], y=b[jj]$. Then $Q_y(x)=0$, so $\{x, Q_y x, a\}=0$.
\end{enumerate}

Thus the claim is verified. The same argument works for $B(y,x)$ by symmetry.
\end{proof}

\begin{cor}
$B(x,y)^{-1} = B(-x,y) = B(x,-y)$.
\begin{proof}
This follows directly from the claim above.
\end{proof}
\end{cor}


\begin{lem}\label{calcberg}
In $\mathcal{H}_3(\biga, \ast)$, we have the following identities for $  1 \leq i \neq j  \neq k\leq 3$ involving Bergmann operators:
\begin{itemize}
\item[\{$\iota$\}]$B(\alpha[ii],a[ij])=B([ii],\alpha a[ij])$

\item[\{I\}]$B(a[ij],b[jk])=B([ij],ab[jk]) =B(ab[ij],[jk])$

\item[\{$\bullet$\}]$B(a[ij],\alpha[jj])=B(a \alpha[ij],[jj])$

\item[\{$\iota \iota$\}] $B(\alpha[ii],a[ij])B(\beta[ii],b[ij]) = B([ii],(\alpha a+ \beta b)[ij])$

\item[\{II\}] $B(a[ij],b[jk])B(c[ij],d[jk]) = B([ij],(ab+cd)[jk]) = B((ab+cd)[ij],[jk])$

\item[\{$\bullet \bullet$\}]$B(a[ij],\alpha[jj])B(b[ij],b[jj]) = B((a \alpha+ b \beta )[ij],[jj])$
\end{itemize}
\begin{proof}
This again involves term by term checks against the multiplication table in Remark \ref{multh3}: The only basis elements that $B(\alpha[ii],a[ij])$ might not act on as identity are:
$b[ji] \mapsto b[ji] -(\alpha ab+ (\alpha ab)^*[ii]$;
$\gamma[jj] \mapsto  \gamma[jj] - \alpha a \gamma [ij] + (\alpha a) \gamma (\alpha a)^\ast[ii]$;
$c[jk] \mapsto c[jk] - \alpha (ac) [ik]$, which justifies $\{\iota\}$ \& $\{\iota \iota\}$;
 The only basis elements that $B(a[ij],b[jk])$ may not act on as identity are:
$\gamma[kk] \mapsto \gamma[kk]- \gamma (ab)^*[ki] + (ab)\gamma (ab)^\ast [ii]$;
$c[ki] \mapsto c[ki] - ((ab)c+c^*(ab)^*)[ii]$;
$d[kj] \mapsto d[kj] - d^* (ab)^* [ji]$, which is sufficient to show \{I\} \& \{II\}. Likewise for $\{\bullet\} \& \{\bullet \bullet \}$
\end{proof}
\end{lem}

The computation above yields:
\begin{cor}
$B(\alpha[ii],s[ij]) = B(\alpha [ik],s[kj])$ and $  B(s[ik],\alpha[kj])= B(s[ij],\alpha[kj])$ holds for $\forall\  1 \leq i \neq j  \neq k\leq 3 $.
\end{cor}

This allows the following "coordinatization" of root subgroups:

\begin{defn}
For any root $\gamma \in C_3$, one may define abelian group isomorphisms such that:
\[
{G_\gamma} \  \text{\large $\cong$}
\begin{dcases}
\biga,  &{\text If} \ \gamma \in \pm\{w_i+w_j\} \\
Sym(\biga), &{\text If} \ \gamma \in \pm\{2w_i\} \\
\biga,  &{\text If} \ \gamma \in \{w_i-w_j\}
\end{dcases}
\]
The isomorphisms in the first two cases are inherited from the matrix structure of $H_3(\biga)$. In the third case, the isomorphism given by
$\beta([ik]^+,s[kj]^-) \mapsto s$ is well-defined, ensured by the preceding lemma and its corollary.

We will also stick to the convention that is

\[ G_{\pm (w_i+w_j)}(a)=\exp_\pm (a[ij])=\exp_\pm (a^*[ji]) \quad  \mbox{for }  i < j.
\]

\end{defn}


\begin{lem}
Denoting by $G_\gamma(s)$  the preimage of $s \in \biga (\text{resp.} \ Sym(\biga))$ under the preceding isomorphisms, the relation $G_\gamma(s)G_\gamma(r)=G_\gamma(s+r)$ holds for for all $\gamma \in C_3$.
\begin{proof}
This is obvious for $ \gamma \in \pm\{\{w_i+w_j\} \bigcup \{2w_k\}\}$. For the remaining roots, this is a direct consequence of Lemma \ref{calcberg}
\end{proof}
\end{lem}

We now compute $[\exp_+(x),\exp_-(y)]$ for the cases (1) $x=\alpha[ii]^+ , y = b[ij]^-$; (2) $x=a[ij]^+ , y = b[jk]^-$, (3) $x=b[ij]^+ , y = \alpha[ii]^-$.
 According to Theorem \ref{c3ln}, it is clear that $[\exp_+(\alpha[ii]),\exp_-(b[ij])] $ is supposed to be equal to

\begin{fleqn}[13pt]
\begin{align*}
\left(
\begin{array}{ccc}
h^+ & 0 & 0\\
h_0 \circ ad_{y'} & h_0 & 0\\
h^- \circ Q_{y'} & h^- \circ ad_{y'} & h^-\\
\end{array}
\right)= 
\underbrace{
\left(
\begin{array}{ccc}
h^+ & 0 & 0\\
0 & h_0 & 0\\
0 & 0 & h^-\\
\end{array}
\right)} _\text{$\in G_{w_i-w_j}$}
\underbrace{
\left(
\begin{array}{ccc}
1 & 0 & 0\\
ad_{y'} & 1 & 0\\
Q_{y'} & ad_{y'} & 1\\
\end{array}
\right)
} _\text{$\in G_{-2w_j}$} 
\end{align*}
\end{fleqn}
for some $h$'s given by a automorphism of $V$ (see our notations below Theorem \ref{c3ln}) and $y' \in V^-$. Direct computation of the left hand side of the above equation tells us that $h^-$ is $B(b[ij],\alpha[ii])^{-1}$ while $Q_y=Q_{(b^*\alpha b)[ii]}$. This gives the relation:

\begin{fleqn}[0pt]
\begin{alignat*}{8}
&(1) \ [\exp_+(\alpha[ii]),\exp_-(b[ij])]\ &=& \ G_{w_i-w_j}(\alpha b)G_{-2w_j}(b^* \alpha b)
\\
&\text{Similar computations yield:}&&
\\
&(2) \ [\exp_+(\alpha[ij]),\exp_-(b[jk]) ]\ &=& \ G_{w_i-w_k}(\alpha b)
\\
&(3)  \ [\exp_-(\alpha[ii]),\exp_+(b[ij]) ]\ &=& \ G_{w_j-w_i}(- b^* \alpha)G_{2w_j}(b^* \alpha b)
\\
&\text{As are:}&&
\\
&(4) \ [G_{w_i-w_j}(a) , G_{w_j-w_k}(b)]\ &=& \ G_{w_i-w_k}(ab)
\\
&(5) \ [G_{w_i-w_j}(a) , G_{w_j+w_k}(b)]\ &=& \ G_{w_i-w_k}(ab)
\\
&(6) \ [G_{w_i-w_j}(a) , G_{-w_i-w_k}(b)]\ &=& \ G_{- w_j-w_k}(-a^*b)
\end{alignat*}
\end{fleqn}

We can also compute the commutator relation involving both long and short root subgroups. Noting that $G_{w_i-w_j}(r)=\beta(r[ij]^+,[jj]^-),\ G_{w_i+w_j}(s)=\exp_+(s[ij])$, the relation

\begin{fleqn}

\begin{alignat*}{8}
&(7) \ [G_{w_i-w_j}(r) , G_{w_i+w_j}(s)]\ &=& \ G_{2w_i}(-rs^*-sr^*)
\end{alignat*}follows from the fact that
\begin{alignat*}{8}
\quad G_{w_i-w_j}(r)^{-1}G_{w_i+w_j}(s)^{-1}G_{w_i-w_j}(r)G_{w_i+w_j}(s)=&\exp_+(\beta(-r[ij]^+,[jj]^-)(-s[ij]))G_{w_i+w_j}(s)\\
=&\exp_+(-s[ij]+\{r[ij],[jj],-s[ij]\}+0)G_{w_i+w_j}(s)\\
=&\exp_+(-s[ij]-rs^*-sr^*)\exp_+(s[ij])\\
=&\exp_+(-rs^*-sr^*)
\end{alignat*}

\end{fleqn}

\begin{rem}
By its commutator relations, one can see that $\text{PE}(V)$ generalizes elementary hyperbolic unitary groups with full form ring and $\epsilon=-1$ (see \cite{HO}) to the setting of alternative rings.
\end{rem}

We need two more definitions 
before we state the main theorem.

\begin{defn}
We denote by $\St(V)$ the group generated by the the set of symbols
$\{ G_\gamma(a) | \ \gamma \in C_3 ,\ a \in \biga \ \text{if $\gamma$ is short} , a \in Sym(\biga) \ \text{if $\gamma$ is long} \}$
satisfying the commutator relations given by the commutators $[G_{\gamma_1}, G_{\gamma_2}] (\gamma_1 \neq -\gamma_2 )$ in $\text{PE}(V)$.
\end{defn}

\begin{rem}
See \cite{ln} for a general exposition for Steinberg groups of Jordan pairs.
\end{rem}

\begin{defn}(See \cite[8.3]{EJK} , \cite{HO})
We say that $Sym(\biga)$ is {\it finitely generated as a form ring} if $\text{there exists a finite set} \ \{a_i\}_{i=1}^k \subset Sym(\biga)$ such that
$$ Sym(\biga)= \{\sum\limits_{i=1}^{k}s_ia_i s_i^* + (r + r^*) | \forall s_i,r \in \biga \} $$
\end{defn}

\begin{thm}
$\St(V)$ is Kazhdan if $Sym(\biga)$ is finitely generated as a form ring.
\end{thm}

\begin{proof}
By the commutator relations of $\text{PE}(V)$ it  follows immediately that the relation is strong, therefore \cite[Theorem 5.1]{EJK} applies.

It remains to show relative property $(T)$ for each root subgroup. The proof of
 \cite[Prop. $8.6(a)$]{EJK} could be applied to our situation verbatim. We summarize the proof here: using the commutator relations we gave previously, the statement is true for any short root subgroup $G_\gamma$ since it lies in a bigger subgroup that is a homomorphic image of $St(\biga)$. For long roots it is a bounded generation argument: If we set for the commutator relations (3) and (7) $\alpha=a_i, \ b=s_i$ and $s=-1$ as in the definition of $Sym(\mathcal{A})$ , then any element in $G_{2w_i}$ can be written as a finite product (of bounded length) of the elements from short root subgroups and their conjugates by the finite set $\{G_\gamma(a_i) \mid {a_i} \mbox{ generates } Sym(\mathcal{A}), \gamma\in C_3 \mbox{ long}\}$. Since $(G,G_{\gamma_{short}})$ and $(G,G_{\gamma_{short}}^{G_{\gamma_{long}}(a_i)})$ are all relative ($T$) pairs, a direct application of  Lemma \ref{bddgen} finishes the proof.
\end{proof}

Since it is a homomorphic image of $\St(V)$, we have:
\begin{cor}
PE$(V)$ is Kazhdan.
\end{cor}

\section*{Acknowledgements}

The author would like to thank Efim Zelmanov and Mikhail Ershov for numerous helpful comments and suggestions.

%
%
%
%

%

\


\end{document}